\documentclass[final,onefignum,onetabnum]{siamart220329}

\usepackage{braket,amsfonts}
\usepackage{array}
\usepackage{bm}
\usepackage{amsmath,amssymb}
\usepackage{mathtools}
\usepackage{stmaryrd}
\usepackage{blkarray}
\usepackage{enumerate}
\usepackage{subcaption}
\usepackage{pgfplots}

\newsiamthm{examples}{Example}
\newsiamthm{prop}{Proposition}
\newsiamthm{claim}{Claim}
\newsiamremark{remark}{Remark}
\newsiamremark{hypothesis}{Hypothesis}
\newsiamremark{assumption}{Assumption}
\crefname{hypothesis}{Hypothesis}{Hypotheses}

\usepackage[noend]{algpseudocode}

\usepackage{footnote}
\usepackage{graphicx,epstopdf}

\usepackage{tabularx}

\Crefname{ALC@unique}{Line}{Lines}
\allowdisplaybreaks

\usepackage{amsopn}
\DeclareMathOperator{\range}{range}

\DeclareMathOperator*{\argmin}{arg\,min}

\newcommand{\R}{\mathbb{R}}

\newcommand{\E}{\mathbb{E}}

\newcommand{\bigO}{\mathcal{O}}

\newcommand{\norm}[1]{\left\lVert#1\right\rVert}

\usepackage[normalem]{ulem}

\newcommand{\ignore}[1]{}

\usepackage{xspace}
\usepackage{bold-extra}
\usepackage[most]{tcolorbox}

\colorlet{texcscolor}{blue!50!black}
\colorlet{texemcolor}{red!70!black}
\colorlet{texpreamble}{red!70!black}
\colorlet{codebackground}{black!25!white!25}

\lstdefinestyle{siamlatex}{%
  style=tcblatex,
  texcsstyle=*\color{texcscolor},
  texcsstyle=[2]\color{texemcolor},
  keywordstyle=[2]\color{texemcolor},
  moretexcs={cref,Cref,maketitle,mathcal,text,headers,email,url},
}

\tcbset{%
  colframe=black!75!white!75,
  coltitle=white,
  colback=codebackground, 
  colbacklower=white, 
  fonttitle=\bfseries,
  arc=0pt,outer arc=0pt,
  top=1pt,bottom=1pt,left=1mm,right=1mm,middle=1mm,boxsep=1mm,
  leftrule=0.3mm,rightrule=0.3mm,toprule=0.3mm,bottomrule=0.3mm,
  listing options={style=siamlatex}
}

\newtcblisting[use counter=example]{example}[2][]{%
  title={Example~\thetcbcounter: #2},#1}

\newtcbinputlisting[use counter=example]{\examplefile}[3][]{%
  title={Example~\thetcbcounter: #2},listing file={#3},#1}

\DeclareTotalTCBox{\code}{ v O{} }
{ 
  fontupper=\ttfamily\color{black},
  nobeforeafter,
  tcbox raise base,
  colback=codebackground,colframe=white,
  top=0pt,bottom=0pt,left=0mm,right=0mm,
  leftrule=0pt,rightrule=0pt,toprule=0mm,bottomrule=0mm,
  boxsep=0.5mm,
  #2}{#1}

\patchcmd\newpage{\vfil}{}{}{}
\begin{tcbverbatimwrite}{tmp_\jobname_header.tex}
\title{Nystr\"om method for symmetric indefinite matrices\thanks{Date: \today 
\funding{YN is supported by EPSRC grants EP/Y010086/1 and EP/Y030990/1. TP is supported by the RandESC project, funded by the Swiss Platform for Advanced Scientific Computing (PASC).}
}
}

\author{
Yijia Chen\thanks{Mathematical Institute, University of Oxford, Oxford, OX2 6GG, UK, (
\email{yijia.chen@balliol.ox.ac.uk},
\email{yuji.nakatsukasa@maths.ox.ac.uk}, \email{anjali.narendran@maths.ox.ac.uk}).}
\and
Yuji Nakatsukasa\footnotemark[2]
\and Anjali Narendran\footnotemark[2] \and Taejun Park\thanks{Institute of Mathematics, EPF Lausanne, 1015 Lausanne, Switzerland (\email{taejun.park@epfl.ch}).}}

\headers{Indefinite Nystr\"om method}{
Yijia Chen, 
Yuji Nakatsukasa, Anjali Narendran, and Taejun Park}
\end{tcbverbatimwrite}
\input{tmp_\jobname_header.tex}

\ifpdf
\hypersetup{pdftitle={Nystr\"om method for symmetric indefinite matrices} }
\fi

\begin{document}
\maketitle

\begin{tcbverbatimwrite}{tmp_\jobname_abstract.tex}
\begin{abstract}
The Nystr\"om method approximates $A\approx A(\,:\,,I)A(I,I)^{\dagger} A(\,:\,,I)^{\top}=CA(I,I)^{\dagger} C^{\top}$, where $C:=A(:,I)\in\mathbb{R}^{n\times r}$ is a column subset matrix of $A$. When applied to symmetric but indefinite matrices, the Nystr\"om method can fail because the core matrix $A(I,I)$ may severely underestimate the eigenvalues of $A$ and may become (nearly) singular. We address this issue by developing and analyzing an algorithm that carefully chooses $\widehat{M}\in\mathbb{R}^{r\times r}$ in place of $A(I,I)^\dagger$ by solving the two-sided sketched least-squares problem $\min_{M}\|X(A-CMC^{\top})X^{\top}\|_F$, where $X\in\mathbb{R}^{t\times n}$ is a random sketch matrix. We study in detail the cases where $X$ is a Gaussian or a leverage score sampling (LSS) matrix, and show that with oversampling $t>r$ the residual $\|A-C\widehat{M}C^{\top}\|_*$ is comparable to $\min_{M}\|A-CMC^{\top}\|_*$. For the Gaussian sketch, we require $t=\mathcal{O}(r)$ samples; for LSS, we show that $t=\mathcal{O}(r \log r)$ samples suffice for the theoretical guarantee, with the LSS approach carrying the advantage that once a set of $t$ row indices is identified, the approximation requires only $t^{2}$ matrix-entry evaluations to find $\widehat{M}$, given $C$. We illustrate our results with synthetic examples and applications to kernel methods.
\end{abstract}


\end{tcbverbatimwrite}
\input{tmp_\jobname_abstract.tex}

\section{Introduction} \label{sect:intro} Real symmetric matrices arise in many contexts, including kernel methods in machine learning~\cite{CER05}. In many practical cases, these kernel matrices exhibit rapidly decaying singular values, making them well-suited for low-rank approximation~\cite{DPM05, AGMWM13}. The Nystr\"om method is well-suited to this regime for symmetric positive semidefinite (SPSD) matrices and has been used frequently by the machine learning community for kernel-based methods \cite{pmlr-v38-anderson15,AGMWM13}. 

Let $A \in \mathbb{R}^{n \times n}$ be an SPSD matrix and let $r<n$ be the target rank, often with $r\ll n$. The matrix $C := A(:,I) \in \mathbb{R}^{n \times r}$ is a column subset matrix of $A$, where $I$ denotes the selected column index set. The classical Nystr\"om method takes the form 
\begin{equation} \label{eqn_nystrom}
    A \approx C\, A(I,I)^{\dagger} C^{\top}.
\end{equation}
This method is both accurate and efficient: its computational cost is $\mathcal{O}(nr + r^3)$ (assuming $I$ is given and using the entry access model, where evaluating an entry of $A$ costs $O(1)$ operations), and the approximation error satisfies
\begin{equation*}
    \|A - C A(I,I)^{\dagger} C^{\top}\|_* \le 
    (r+1)\|A-A_r\|_*
    =(r+1)\bigl(\sigma_{r+1}(A) + 
    \cdots + \sigma_n(A)\bigr)
\end{equation*} for an appropriately chosen index set $I$~\cite{osinsky2023closelaa,cortinovis2024adaptive}; here $A_r$ is the rank-$r$ truncated SVD, i.e., the optimal rank-$r$ approximation. The above result shows that for any SPSD matrix, there is a Nystr\"om approximation whose error is within a factor $r+1$ of the truncated SVD of the same rank. 

However, in practice, many important applications involve symmetric \emph{indefinite} matrices. Examples include the jittering and tangent distance kernels~\cite{indefinitekernel}, as well as low-rank approximations arising in natural language processing~\cite{devlin-etal-2019} and in non-metric similarity or dissimilarity measures for social networks~\cite{GISBRECHT2015643}. Applying the standard Nystr\"om method to indefinite matrices is problematic: the submatrix $A(I,I)$ may severely underestimate the eigenvalues of $A$, making it nearly singular, thereby causing instability and unbounded approximation errors~\cite{Nakatsukasa2023}.

In prior work \cite{Nakatsukasa2023}, the \emph{indefinite Nystr\"om method} was proposed, which uses random embeddings and incorporates oversampling to, say, $1.5$ times the target rank, followed by truncation of the core matrix back to the target rank to stabilize the approximation. When $A$ is sparse, that approach employs a random embedding with computational cost $\mathcal{O}(\xi\, \mathrm{nnz}(A) + r^3)$, where $\xi = 8$. For dense matrices, a subsampled randomized trigonometric transform (SRTT) sketch is used, achieving total cost $\mathcal{O}(n^2 \log r + r^3)$. While the algorithm of \cite{Nakatsukasa2023} performs robustly in numerical experiments, its theoretical analysis does not correspond to the proposed algorithm. Moreover, in that approach, $X$ needs to be dense, and therefore the entire matrix $A$ needs to be accessed, making the approach expensive. This is overcome in this paper by using a leverage score sampling (LSS) matrix.

In this work, we introduce a simpler and more efficient approach for constructing the middle matrix in \eqref{eqn_nystrom}. Our method also uses sketching and oversampling but avoids some of the computational overhead of previous algorithms. We establish theoretical performance bounds and demonstrate the effectiveness of our algorithm through numerical experiments. Let us emphasize that $C$ is assumed to be given, and the goal is to find a good $r\times r$ matrix $M$ such that $CMC^\top\approx A$; the choice of $C$ is clearly a separate matter, and a plethora of methods are now available, including 
sketch-and-pivot-based methods~\cite{dong2023simpler}, 
recent methods with stronger theoretical guarantees while maintaining practical speed~\cite{cortinovis2024adaptive,osinsky2023closelaa}, and range-finding algorithms~\cite{halko2010radsvd} (in which case $C$ is not restricted to be a subset of columns).

\subsection{Related work} \label{subsec:related}
Like the Nystr\"om method, many existing methods for low-rank approximation of a symmetric matrix $A$ rely on projections onto an approximate column space. These techniques often yield approximations of the general form $A \approx CC^\dagger A (C^{\top})^\dagger C^{\top}$, where the columns of $C$ approximate the column space of $A$. A well-known example is randomized SVD, which can be expressed as $A \approx QQ^{\top}AQQ^{\top}$, with $Q = \operatorname{orth}(AX^{\top})$ for a random sketching matrix $X$ \cite{halko2010radsvd}. 
A direct application of Osinsky's algorithm~\cite{osinsky2023closelaa} yields the bound $\|A - CC^\dagger A (C^{\top})^\dagger C^{\top}\|_F \leq \sqrt{2(1+r)} \|A-VV^{\top}A\|_F$, where $V \in \mathbb{R}^{n\times r}$ is the rank-$r$ column space approximator used to obtain $C$. CUR matrix decompositions provide a related framework that also yields interpretable low-rank approximations \cite{mahoneya2009CUR} and allows for fast approximation of parameter-dependent matrices~\cite{park2025low}. Our proposed method using two-sided sketching shares the general $CMC^{\top}$ form, as shown in \Cref{subsec:twosidedls}, but allows for computation based on subsampling $A$, thereby having complexity strictly lower than $O(n^2)$, and performs most favorably among all existing methods in our empirical comparisons in \Cref{sec:experiments}.

\subsection{List of contributions} 
To approximate a real symmetric matrix $A \approx CMC^{\top}$, we propose a faster approach based on solving a sketched two-sided least-squares problem. We show that this approach produces a middle matrix $M$ that is theoretically near-optimal and also performs competitively with existing methods.

Instead of solving the full problem
\[M^*:=\underset{M\in\mathbb{R}^{r\times r}}\argmin\| A-CMC^{\top} \|_{F}, 
\]
we solve the sketched problem
\[\widehat{M} := \underset{M\in \mathbb{R}^{r\times r}}{\text{argmin}}\|X (A-CMC^{\top})X^{\top}\|_{F},
\]
and obtain the following relative error bounds.
\begin{enumerate}
    \item \emph{General bound.} \Cref{thm:general} establishes, for any sketching matrix $X$ with orthonormal rows and any unitarily invariant norm $\|\cdot\|$,
    \[
    \|A - C\widehat{M}C^{\top}\|\leq \left( 1+ \frac{1}{\sigma_{\min}(XQ)^2}\right)\|A - C M^* C^{\top}\|,
    \] 

    \item \emph{Gaussian sketch.} \Cref{main_gaussian} analyzes the case $G \in \mathbb{R}^{t \times n}$ with $t = \mathcal{O}(r)$, and shows an $n$-independent nuclear-norm bound
    \[
    \|A - C\widehat{M}C^{\top}\|_* \leq \left(1+c_2(r,t,\epsilon)\sqrt{r}\right)\|A - C M^* C^{\top}\|_*,
    \]
    where $c_2$ is an order-1 constant. Here, the sketched approximation achieves near-optimal accuracy, differing from the optimal approximation by a factor that is at most on the order of $\sqrt{r}$. 
    
    \item \emph{Leverage score sampling.} \Cref{main_leverage} considers $S \in \mathbb{R}^{t \times n}$ with $t = \mathcal{O}(r \log r)$ and proves that
    \[
    \|A - C\widehat{M}C^{\top}\|_* \leq (1+c_5(r,t,\epsilon)\sqrt{r})\|A - C M^* C^{\top}\|_*,
    \]
    where $c_5$ is an order-1 constant. LSS is computationally efficient: when an index set is identified, the low-rank approximation can be computed using no more than $t^2$ additional entry evaluations.
\end{enumerate}

\subsection{Notation} Throughout the paper $\|\cdot \|_2$ is the spectral norm of a matrix (or the vector $\ell_2$-norm), $\|\cdot\|_F$ is the Frobenius norm, and $\| \cdot\|_*$ is the nuclear (or trace) norm. We use $\dagger$ for the Moore--Penrose pseudoinverse~\cite{pseudoinverse1976}. Unless otherwise specified, $\sigma_{i}(\cdot)$ refers to the $i$th largest singular value and $\sigma_{\min}$ refers to the smallest, i.e., for an $n\times r (n\geq r)$ matrix $C$, $\sigma_{\min}(C)=\sigma_r(C)$. The Kronecker delta is $\delta_{ij} = \mathbb{I}_{i=j}$, while $\delta$ without subscripts is a probability parameter. Constants $c_k(\cdot)$ denote order-1 quantities depending on their arguments. The $i$th canonical basis vector of $\mathbb{R}^{n}$ is $e_{i}$. For a matrix $B$, $B_{ij}$ refers to the $(i,j)$ entry, $B_{i,*}$ the $i$th row, and $B_{*, i}$ the $i$th column. We say that an $n\times r (n\geq r)$ matrix $U$ is orthonormal if its columns are orthonormal.

Throughout, $A$ denotes an $n\times n$ real symmetric matrix for which we seek a low-rank approximation. The matrix $C$ is an $n\times r$ matrix approximating the span of $A$, with left singular vector matrix $U$; in practical situations $C$ is often a column subset of $A$. $M$ is the middle matrix in the approximation $A\approx CMC^\top$, with $M^*$ and $\widehat{M}$ denoting the optimal and suboptimal middle matrices in the full and sketched problems, respectively. For the $t\times n$ sketch matrix, we distinguish between $X$, $G$, and $S$: $X$ denotes a general subspace embedding, $G$ denotes a Gaussian matrix with entries i.i.d. $\mathcal{N}(0,1/t)$, and $S$ denotes a leverage score sampling (LSS) matrix.

\section{Preliminaries} This section introduces the mathematical preliminaries necessary for the subsequent proofs.

\subsection{Sketching and subspace embedding}
Sketching is a key tool in randomized numerical linear algebra, and the subspace embedding property forms the theoretical foundation for sketching.

\begin{definition}[$\ell_2$-subspace embedding] A $(1\pm \epsilon)$ $\ell_2$-subspace embedding for the column space of $A\in \Bbb{R}^{n\times d}$ $(n>d)$ is a matrix $X\in \mathbb{R}^{t\times n}$ $(d\le t<n)$ such that for all $x\in \Bbb{R}^{d}$ \cite{SketchingAsATool},
\[
(1- \epsilon)\| Ax\|_2^2\leq \| XAx\|_2^2 \leq (1+ \epsilon)\| Ax\|_2^2.
\]    
\end{definition}

\begin{lemma}[\cite{SketchingAsATool}] \label{boundsingularvalue1} Let $U\in \Bbb{R}^{n\times d}$ have orthonormal columns, and let $X\in \Bbb{R}^{t\times n}$ be a $(1\pm \epsilon)$ $\ell_2$-subspace embedding of $U$. Then
\begin{enumerate}[(i)]
    \item $\|XUy\|_2^2= (1\pm \epsilon)\| Uy\|_2^2=(1\pm \epsilon)\| y\|_2^2$ for all $y\in \Bbb{R}^{d}$,
    \item $\|U^{\top}X^{\top}XU-I_d\|_2\leq \epsilon$. 
\end{enumerate}
\end{lemma}

If such an embedding can be constructed without knowledge of $U$ beyond its dimension, it is called \emph{oblivious}; otherwise it is \emph{non-oblivious}.

\begin{definition}[Oblivious $\ell_2$-subspace embedding \cite{SketchingAsATool}] Suppose $\mathcal{D}$ is a distribution on $t\times n$ matrices $X$, where $t$ is some function of $n, d, \epsilon, \delta$. If with probability at least $1-\delta$, for any fixed $n\times d$ matrix $U$, a matrix $X$ drawn from $\mathcal{D}$ is a $(1\pm \epsilon)$ $\ell_2$-subspace embedding for $U$, then $\mathcal{D}$ is called an $(\epsilon, \delta)$ oblivious $\ell_2$-subspace embedding.
\end{definition}

A number of classes of (random) matrices are known to satisfy the oblivious subspace embedding property. These include Gaussian matrices, subsampled randomized trigonometric transforms (SRTTs), and sparse embeddings~\cite{MartinssonTroppacta}.
By contrast, leverage score sampling also yields an $\epsilon$-subspace embedding with $t=O(d\log d)$~\cite[Ch.~2]{SketchingAsATool}, but is non-oblivious; it uses information on $U$, namely its row-norms. 

\subsection{Solutions of symmetric two-sided least-squares problems} \label{subsec:twosidedls} 
Given a column subset $C \in \mathbb{R}^{n\times r}$ of a real symmetric matrix $A\in \mathbb{R}^{n\times n}$, the optimal middle matrix $M^*\in \mathbb{R}^{r\times r}$ is 
\begin{equation} \label{fullprob}
M^*:=\underset{M\in\mathbb{R}^{r\times r}}{\text{argmin}}\| A-CMC^{\top} \|_{F}= C^{\dagger}A(C^{\top})^{\dagger} = (C^{\top}C)^{-1}C^{\top}AC(C^{\top}C)^{-1}.
\end{equation}

Writing $C:=QR$ as the thin QR decomposition of $C$ and substituting gives
\begin{equation} \label{CURBA_full}
    M^* = R^{-1}(Q^{\top}AQ)R^{-\top}.
\end{equation}

The computation of~\eqref{CURBA_full} provides the best possible approximation for any given column subset $C$. However, evaluating $M^*$ via \eqref{CURBA_full} requires computing $Q^{\top}AQ$ at cost $\mathcal{O}(n^2 r)$ and evaluation of all $n^2$ entries of $A$, which is prohibitive when $n$ is large.

Instead, consider the sketched problem for a symmetric matrix $A\in \mathbb{R}^{n\times n}$,
\begin{align}
\widehat{M} :=& \underset{M\in \mathbb{R}^{r\times r}}{\text{argmin}}\|X(A-CMC^{\top})X^{\top}\|_{F} 
= (XC)^\dagger XAX^{\top} ((XC)^{\top})^\dagger \label{sketchprob} \\
=& (C^{\top}X^{\top}XC)^{-1}C^{\top}X^{\top}XAX^{\top}XC\left(C^{\top}X^{\top}XC\right)^{-1}\nonumber,
\end{align}
where $X\in \mathbb{R}^{t\times n}$ is a \textit{sketching matrix} with $r\leq t<n$ (usually $t \ll n$), assumed such that $XC$ has full column rank. Writing $C := QR$ and $XQ := \widetilde{Q}\widetilde{R}$,
\begin{equation} \label{CURBA_sketch}
\widehat{M}=(\widetilde{R}R)^{-1}(\widetilde{Q}^{\top} XAX^{\top}\widetilde{Q})(\widetilde{R}R)^{-\top}.
\end{equation}
Multiplying by $X$ on the left and $X^{\top}$ on the right reduces the effective matrix size from $n\times n$ to $t\times t$: instead of computing $Q^{\top}AQ$, one computes $\widetilde{Q}^{\top}XAX^{\top}\widetilde{Q}$ at cost $\mathcal{O}(t^2r)$ whenever $XAX^{\top}$ is cheap to form; for example in the entry access model and when $X$ is a subsampling matrix, $XAX^{\top}$ is a $t\times t$ principal submatrix of $A$ and can be formed using $t^2$ entry evaluations. 

The sketched problem \Cref{sketchprob} is efficient. The main goal of the subsequent sections is to show that $\widehat{M}$ performs nearly as well as $M^*$, making the sketched approximation near-optimal for all real symmetric matrices, including indefinite ones.

\section{Proposed algorithm}
Following \Cref{subsec:twosidedls}, we propose an efficient and reliable way of computing $C\widehat{M}C^{\top}$ given $C$. \Cref{alg:symmetric_cur} presents our general framework for computing the Nystr\"om method for symmetric indefinite matrices. Note that the output of \Cref{alg:symmetric_cur}, $C\widetilde{M}C^\top$, is the same as $C\widehat{M}C^\top$.

\begin{algorithm}[H]
\caption{Indefinite Nystr\"om Method}\label{alg:symmetric_cur}
 \hspace*{\algorithmicindent} \textbf{Input} A symmetric (indefinite) matrix $A \in \R^{n\times n}$, a column subset $C \in \R^{n\times r}$, sketch size $t$ (e.g., $t = 2r$ for Gaussian, $t = \lceil 2r\log r \rceil$ for LSS).\\
 \hspace*{\algorithmicindent} \textbf{Output} Middle matrix $\widetilde{M}\in \R^{r\times r}$ such that $A\approx C\widetilde{M}C^\top$.
\begin{algorithmic}[1]
\Procedure{Indef\_Nystr\"om}{$A,C,t$}
  \State Draw a subspace embedding $X \in \R^{t\times n}$ for $\operatorname{span}(C)$, e.g., Gaussian or LSS,
  \State $XC \gets X \cdot C$, $M \gets X\cdot  A \cdot X^\top$,
  \State $[Q,R] \gets \operatorname{qr}(XC)$,
  \State $\widetilde{M} \gets R^{-1}(Q^\top MQ)R^{-\top}$,
\EndProcedure
\end{algorithmic}
\end{algorithm}

\paragraph{Complexity} The costs of \Cref{alg:symmetric_cur} are as follows. First, if $C$ is formed as a column subset of $A$, this requires $nr$ entry evaluations. For the Gaussian sketch, it requires $\bigO(nt)$ to form the embedding (line $2$), and $\bigO(ntr)$ and $\bigO(n^2t)$ for forming $XC$ and $M$ (line $3$), respectively, along with all $n^2$ entry evaluations, whereas leverage score sampling (LSS) requires $\bigO(nr^2)$ to form the embedding, and $t^2$ entry evaluations to form $XC$ and $M$. Line $4$ costs $\bigO(tr^2)$, while line $5$ costs $\bigO(t^2r)$ to form $Q^\top M Q$ and $\bigO(r^3)$ to apply $R^{-1}$ and $R^{-\top}$.

Summing up, the Gaussian sketch requires all $n^2$ entry evaluations and $\bigO(n^2t)$ flops with $t = \bigO(r)$, whereas LSS requires $nr+t^2$ entry evaluations and $\bigO(nr^2+t^2r)$ flops with $t = \bigO(r\log r)$. In particular, the LSS cost is linear in $n$, both in flops and in entry evaluations, which is the key practical advantage.

In the following sections, our goal is to prove that the algorithm yields a reliable and good enough approximation for $M$ with theoretical guarantees.

\section{Accuracy of the indefinite Nystr\"om method}
\label{sec:theory}
In this section, we present our main theoretical results. We first present a relative nuclear-norm bound in \Cref{sec:generalbound}, which holds under generic assumptions on the subspace embedding. Then we specialize to Gaussian and leverage score sampling in \Cref{sec:gaussian,sec:lssampling}, respectively.

\subsection{General theoretical bounds for the sketched solution} 
\label{sec:generalbound}

\begin{theorem}\label{thm:general}
Let $A\in \Bbb{R}^{n\times n}$ be a real symmetric matrix, $C\in \Bbb{R}^{n\times r}$ a full-rank column subset of $A$ with $r<n$, with thin QR factorization $C = QR$, and $X\in\mathbb{R}^{t\times n}$ any sketch matrix with orthonormal rows such that $XQ$ has full column rank. For any unitarily invariant norm $\|\cdot \|$, $\widehat M$ in~\eqref{sketchprob} satisfies
\begin{align*}
 \|A-C\widehat{M}C^{\top}\|  \leq  \left(1+ \frac{1}{\sigma_{\min}(XQ)^2}\right)\|A-CM^*C^{\top}\|.
\end{align*}
Furthermore, in the Frobenius norm
\begin{align*}
      \|A-C\widehat{M}C^{\top}\|_F \leq \left(1+ \frac{1}{\sigma_{\min}(XQ)^4}\right)^{\frac{1}{2}}\|A-CM^*C^{\top}\|_F.
\end{align*}
\end{theorem}

\begin{proof} Since $XQ \in \mathbb{R}^{t \times r}$ has full column rank, $(XQ)^{\dagger}$ is a left inverse. Define projections $\mathcal{P}_{Q,X^{\top}}:=Q(XQ)^{\dagger}X$ and $\mathcal{P}_Q:=QQ^{\top}$ (which is an orthogonal projection), so that $CM^*C^\top = \mathcal{P}_QA\mathcal{P}_Q$ and $C\widehat{M}C^\top = \mathcal{P}_{Q,X^{\top}} A \mathcal{P}_{Q,X^{\top}}^\top$. One verifies:
\begin{itemize}
    \item $\mathcal{P}_{Q,X^{\top}}\mathcal{P}_Q = Q(XQ)^{\dagger}XQQ^{\top} = QQ^{\top} = \mathcal{P}_Q$, and so $\mathcal{P}_Q \mathcal{P}^{\top}_{Q,X^{\top}} = \mathcal{P}_Q$;
    \item $\mathcal{P}_Q\mathcal{P}_{Q,X^{\top}} = QQ^{\top}Q(XQ)^{\dagger}X = \mathcal{P}_{Q,X^{\top}} $, and so $\mathcal{P}^{\top}_{Q,X^{\top}}\mathcal{P}_Q = \mathcal{P}^{\top}_{Q,X^{\top}}$.
\end{itemize}
Therefore, for any unitarily invariant norm,
\begin{align*}
       \|A - C\widehat{M}C^{\top}\|
       =&\|A - \mathcal{P}_{Q,X^{\top}}A\mathcal{P}^{\top}_{Q,X^{\top}}\| \\
       =& \| (A - \mathcal{P}_QA\mathcal{P}_Q) + (\mathcal{P}_{Q,X^{\top}}\mathcal{P}_QA\mathcal{P}_Q \mathcal{P}^{\top}_{Q,X^{\top}} - \mathcal{P}_{Q,X^{\top}}A\mathcal{P}^{\top}_{Q,X^{\top}})\|
       \\
        \leq& \| (A - \mathcal{P}_QA\mathcal{P}_Q)\|+ \|\mathcal{P}_{Q,X^{\top}}(A - \mathcal{P}_QA\mathcal{P}_Q)\mathcal{P}_{Q,X^{\top}} \|
        \\
       \leq& (1+ \|\mathcal{P}_{Q,X^{\top}} \|_2^2) \| A - \mathcal{P}_QA\mathcal{P}_Q\|  \\
       =& (1+ \|(XQ)^{\dagger} \|_2^2)\|A-CM^*C^{\top}\| \\ =& \left(1+ \frac{1}{\sigma_{\min}(XQ)^2}\right)\|A-CM^*C^{\top}\|.
\end{align*}
For the improved bound in the Frobenius norm, using $\mathcal{P}_{Q}^{2}=\mathcal{P}_{Q}$ and the cyclicity of the trace,
\begin{align*}
  \|A-C\widehat{M}C^{\top}\|_{F}^{2}
  &=\|A-CM^{*}C^{\top}\|_{F}^{2}
   +\|\mathcal{P}_{Q,X^{\top}}(A-\mathcal{P}_{Q}A\mathcal{P}_{Q})
      \mathcal{P}_{Q,X^{\top}}\|_{F}^{2}\\
  &\quad+2\operatorname{tr}\!\bigl((A-\mathcal{P}_{Q}A\mathcal{P}_{Q})
     (\mathcal{P}_{Q}A\mathcal{P}_{Q}
      -\mathcal{P}_{Q,X^{\top}}A\mathcal{P}_{Q,X^{\top}}^{\top})\bigr)\\
  &=\|A-CM^{*}C^{\top}\|_{F}^{2}
   +\|\mathcal{P}_{Q,X^{\top}}(A-\mathcal{P}_{Q}A\mathcal{P}_{Q})
      \mathcal{P}_{Q,X^{\top}}\|_{F}^{2},
\end{align*}
where the cross-term vanishes by cyclicity of the trace. Hence, the result follows by
\begin{align*}
    \|A-C\widehat{M}C^{\top}\|_{F}^{2}
  &\le\Bigl(1+\|(XQ)^{\dagger}\|_{2}^{4}\Bigr)\|A-CM^{*}C^{\top}\|_{F}^{2} \\
  &=\Bigl(1+\tfrac{1}{\sigma_{\min}(XQ)^{4}}\Bigr)\|A-CM^{*}C^{\top}\|_{F}^{2}.
\end{align*}
\end{proof}

The bounds in \Cref{thm:general} depend on $\sigma_{\min}(XQ)$. These bounds are often overestimates, and may be expensive to determine efficiently. We therefore specialize to two concrete choices of $X$---Gaussian matrices and leverage score sampling---in \Cref{sec:gaussian,sec:lssampling}, where we obtain $n$-independent nuclear-norm bounds. 

\subsection{Random Gaussian matrix} \label{sec:gaussian}
Among oblivious embeddings, a Gaussian embedding is the most fully understood. We choose $G\in \Bbb{R}^{t\times n}$ $(t \leq n)$ with independent and identically distributed (i.i.d.) entries $G_{ij}\sim \mathcal{N}(0, \frac{1}{t})$, so that $\mathbb{E}[G^{\top}G]=I_n$ and $\operatorname{Var}(G^{\top}G)=\frac{1}{t}I_n+\frac{1}{t}J_n$, where $J_n$ is the all-ones matrix. Gaussian matrices have the nice property that they are invariant under orthonormal transformations, that is, given an orthonormal $U\in \Bbb{R}^{n\times r}$ with $n\geq r$, $GU\in\mathbb{R}^{t\times r}$ is also Gaussian with i.i.d. entries $\mathcal{N}(0, \frac{1}{t})$. This allows us to prove the following lemma which is central to our theory.

\begin{lemma} \label{lemma32}
Let $U \in \mathbb{R}^{n\times r}$ be orthonormal and $G \in \mathbb{R}^{t\times n}$ a Gaussian sketching matrix with i.i.d. entries $\mathcal{N}(0, \frac{1}{t})$, where $t> r$. Set 
\begin{equation}\label{eq:c1def}
    c_1(r,t):=  \left(1+\sqrt{\frac{r}{t}}\right)^2 \left(1+\sqrt{\frac{2}{r}}\right)^{1/2} \left(1+\frac{2}{t}\right)^{1/2}.
\end{equation}
Then for any symmetric $B\in \mathbb{R}^{n\times n}$ and $\Delta> c_1(r,t)\sqrt{r}$, 
\[
\mathbb{P}\left( \|U^{\top}G^{\top}GBG^{\top}GU\|_* < \Delta \|B\|_*\right) > 1-\frac{c_1(r,t)\sqrt{r}}{\Delta}.
\]
\end{lemma}
\begin{remark}[$c_{1}$ is order-1]
    Since $t\geq r\geq 1$, $c_1(r,t)\le 4\sqrt{1+\sqrt{2}}\sqrt{3}$. This confirms that $c_1$ is an order-1 constant. Increasing $t$ strictly decreases $c_1(r,t)$, exhibiting the benefit of oversampling.    
\end{remark}

\begin{proof}
We bound
\begin{align*}
    \mathbb{E} \| U^{\top}G^{\top}GBG^{\top}GU\|_* &\leq \sqrt{r}\mathbb{E}\left[ \| GU\|_2^2 \|GBG^{\top}\|_F\right]\\
    & \leq \sqrt{r} \sqrt{\E\|GU\|_2^4 \E\|GBG^\top\|_F^2} \\
    & = \sqrt{r} \left(1+\sqrt{\frac{r}{t}}\right)^2 \left(1+\sqrt{\frac{2}{r}}\right)^{1/2} \sqrt{\text{tr}(B \cdot \mathbb{E}\left[G^{\top}GBG^{\top}G\right])},
\end{align*} where we use $\norm{\cdot}_* \le \sqrt{r}\norm{\cdot}_F$ for the first inequality, Cauchy--Schwarz for the second, and \cite[Lemma B.1]{tropp2023} together with linearity of the trace for the final inequality.

A direct calculation gives $\mathbb{E}[G^\top GBG^\top G] = (1+1/t)B + (\operatorname{tr}(B)/t)I_n$, so
\begin{equation*}
    \operatorname{tr}(B\cdot \mathbb{E}[G^\top GBG^\top G]) = \left(1+\frac{1}{t}\right)\norm{B}_F^2 + \frac{1}{t}\operatorname{tr}(B)^2 \leq \left(1+\frac{2}{t}\right)\norm{B}_*^2.
\end{equation*}
Combining and applying Markov's inequality yields the stated result.
\end{proof}

\begin{lemma}{\cite[Fact 35]{ClarksonWoodruff}} \label{normaleq} Given $C\in \Bbb{R}^{n\times d}$ with $n>d$ and $D \in \Bbb{R}^{n\times n}$, let $W^*:=\underset{W\in \Bbb{R}^{d\times d}}{\text{argmin}}\|CWC^{\top} -D\|_F^2=C^\dagger D (C^{\top})^\dagger$.
Then $C^{\top}(CW^*C^{\top}-D)C=0$.
\end{lemma}

Using the two lemmas above (\Cref{lemma32,normaleq}), we prove the main result for Gaussian embeddings (\Cref{main_gaussian}).

\begin{theorem}\label{main_gaussian} Let $A\in \Bbb{R}^{n\times n}$ be a real symmetric matrix, $C\in \Bbb{R}^{n\times r}$ a column subset with $n\ge r$, $t$ a scalar multiple of $r$ (e.g., $t =2r$), and $G \in \Bbb{R}^{t\times n}$ a Gaussian $(\epsilon, \delta)$-oblivious subspace embedding. Define 
\[ M^*:=\underset{M\in \Bbb{R}^{r\times r}}{\text{argmin}}\| A-CMC^{\top}\|_F^2, \quad \quad \widehat{M}:=\underset{M\in \Bbb{R}^{r\times r}}{\text{argmin}}\| G(A-CMC^{\top})G^{\top}\|_F^2.
\]
Then for all $\Delta> c_1(r,t)\sqrt{r}$,
\[
\| A-C\widehat{M}C^{\top}\|_* \leq  (1+c_2(r,t,\epsilon)\sqrt{r}) \|A-CM^*C^{\top}\|_*,
\]
with probability at least $1-(c_1(r,t)\sqrt{r}/\Delta + \delta)$, where $c_1(r,t)$ is as in \Cref{lemma32}, and $c_2(r,t,\epsilon) :=\Delta/((1-\epsilon)^2\sqrt{r})$.
\end{theorem}

\begin{proof}
Let $C=U\Sigma V^{\top}$ be the SVD of $C$, with $U\in \Bbb{R}^{n\times r}$ orthonormal. Setting $W = \Sigma V^\top MV\Sigma \iff M = V\Sigma^{-1}W\Sigma^{-1}V^\top$ shows
\begin{equation*}
    M^* \iff W^* := \argmin\limits_{W} \norm{A-UWU^\top}_F^2, 
\end{equation*} and
\begin{equation*}
    \widehat{M} \iff \widehat{W}:= \argmin\limits_{W} \norm{G(A-UWU^\top)G^\top}_F^2.
\end{equation*}

By \Cref{normaleq}, $U^{\top}G^{\top}(GAG^{\top}-GU\widehat{W}U^{\top}G^{\top})GU=0$. Substituting $B:=A-UW^*U^{\top}$ in \Cref{lemma32}, with probability at least $1-c_1(r,t)\sqrt{r}/\Delta$,
\begin{align*}
    \| U^{\top}G^{\top}GU(\widehat{W}-W^*)U^{\top}G^{\top}GU\|_* &= \| U^{\top}G^{\top}G(A-UW^*U^{\top})G^{\top}GU \|_* \\ &< \Delta\| A-UW^*U^{\top} \|_*.
\end{align*}

Let $K = U^\top G^\top GU$. Using \Cref{boundsingularvalue1} with $G$ an $\epsilon$-subspace embedding (with probability $\geq 1-\delta$), $\norm{K-I_r}_2 \leq \epsilon<1$, so $K$ is invertible with $\norm{K^{-1}}_2 \leq (1-\epsilon)^{-1}$. Therefore
\begin{align*}
    \|\widehat{W}-W^*\|_* &= \norm{K^{-1}K(\widehat{W}-W^*)KK^{-1}}_* \leq \norm{K^{-1}}_2^2 \norm{K(\widehat{W}-W^*)K}_* \\
    &\leq (1-\epsilon)^{-2}\norm{U^\top G^\top GU(\widehat{W}-W^*)U^\top G^\top GU}_* \\ &< \frac{\Delta}{(1-\epsilon)^2} \norm{A - UW^*U^\top}_*
\end{align*}

Finally,
\begin{align*}
    \| A-U\widehat{W}U^{\top}\|_* &\leq \| A-UW^*U^{\top}\|_* + \| \widehat{W}-W^*\|_* \\
    &\leq  \left( 1+\frac{\Delta}{(1-\epsilon)^2} \right)\|A-UW^*U^{\top}\|_*,
\end{align*}
and a union bound over the two failure events completes the proof.
\end{proof}

While Gaussian sketching provides clean theory, the dense matrix $G$ makes computing $GAG^\top$ expensive, requiring $\mathcal{O}(n^2t)$ operations and access to all $n^2$ entry evaluations of $A$. As a result, Gaussian sketching is often less favorable in practice. In contrast, leverage score sampling, studied next, avoids these computational burdens.

\subsection{Leverage score sampling} \label{sec:lssampling}
The column statistical leverage score quantifies how important each row of $C$ is in approximating $C$. Concretely, let $C = QR$ be a thin QR factorization, so that the columns of $Q \in \R^{n\times r}$ are orthonormal. The $i$th leverage score is $\ell_i := \norm{Q_{i,*}}_2^2$, the squared norm of the $i$th row of $Q$; it is independent of the choice of $Q$ and satisfies $\sum_{i = 1}^n \ell_i = r$. Leverage score sampling draws $t$ row indices i.i.d. with probabilities $p_i:= \ell_i/r$ and assembles $S \in \R^{t\times n}$ whose $m$th row is $e_{j_m}^\top/\sqrt{tp_{j_m}}$, where $j_m$ is the index drawn on the $m$th trial; this scaling ensures $\E[S^\top S] = I_n$. Here, $e_j$ is the $j$th canonical basis vector of $\R^n$. 
Since $S$ is sparse and simply subsamples and rescales rows and columns, forming $SAS^\top$ is efficient in the entry access model, requiring only $t^2$ entry evaluations of $A$.

We now present lemmas leading to the main LSS theorem (\Cref{main_leverage}): (i) a fourth-moment bound on $\norm{SU}_2$ (\Cref{concentration_bound}), (ii) a bound for the product of two independent sketches (\Cref{diffterms}), and (iii) the main nuclear-norm bound (\Cref{lem:nuclearLSS}), obtained by a decoupling argument~\cite[Thm.~6.1.1]{Vershynin_2018} that requires $t = \bigO(r\log r)$ samples.

\begin{lemma} \label{concentration_bound}
Let $U\in \mathbb{R}^{n\times r}$ be orthonormal and $S\in \mathbb{R}^{s\times n}$ the row leverage score sampling matrix of $U$ with $s< n$. Then $\mathbb{E}\|SU\|_2^4 \leq c_3(r,s)$, where $c_3(r,s)$ equals
\begin{equation*}
    \left( \sqrt{2e\frac{r-1}{s}\log(2r)}+4e\frac{r+1}{s}\log(2r) \right)^2 + 2 \left( \sqrt{2\frac{r-1}{s}\log(2r)}+\frac{1}{3}\frac{r+1}{s}\log(2r) \right) +1.
\end{equation*} Taking $s = \bigO(r\log r)$ makes $c_3(r,s)$ an $\bigO(1)$ constant. For example, $s = 2(r+1)\log (2r)$ gives $c_3(r,s) \leq 54$.
\end{lemma}
\begin{proof}
 Write $Z := U^{\top}S^{\top}SU-I_r$, so $\| SU\|_2^4 \leq (\|Z\|_2+1)^2 = \|Z\|_2^2 + 2\|Z\|_2+1$. Define $Z_m := (sp_{i(m)})^{-1} U_{i(m), *}^\top U_{i(m), *} - s^{-1}I_r$, so $Z = \sum_{m=1}^{s}Z_m$. One computes $\mathbb{E}[Z_m] = 0$, $L:=\max_i\norm{Z_i}_2 \leq (r+1)/s$, and variance $\operatorname{Var}(Z) = (r-1)/s$. Matrix Bernstein's inequality \cite[Thm. 6.1.1]{tropp2015matrixconcentration} gives $\mathbb{E}\norm{Z}_2 \leq \sqrt{2\operatorname{Var}(Z)\log(2r)} +\frac{1}{3}L\log(2r)$ and \cite[Thm. A.1]{cgt12} gives the following bound 
 \begin{equation*}
     \mathbb{E}\norm{Z}_2^2 \leq \left( \sqrt{2e\nu(Z)\log(2r)}+4eL\log(2r) \right)^2.
 \end{equation*} Combining these bounds yields $c_3(r,s)$ as stated.
\end{proof}

\begin{lemma}\label{diffterms}
Let $U \in \mathbb{R}^{n\times r}$ be orthonormal, $S_1 \in \R^{s_1 \times n}$ and $S_2 \in \R^{s_2 \times n}$ be two independent leverage score sampling matrices of the rows of $U$, and let $B \in \R^{n\times n}$ be real symmetric. Then $\E\norm{U^\top S_1^\top S_1 BS_2^\top S_2 U}_F \le (c_3(r,s_1)c_3(r,s_2))^{1/4}\norm{B}_F$.
\end{lemma}
\begin{proof} 
By the generalized H\"older inequality and \Cref{concentration_bound},
\begin{align*}
    \mathbb{E}\|U^{\top}S_1^{\top}S_1BS_2^{\top}S_2U\|_F &\leq \left(\mathbb{E}\|S_1U\|_2^4\right)^{1/4}\sqrt{ \mathbb{E}\|S_1BS_2^{\top}\|_F^2}\left(\mathbb{E}\|S_2U\|_2^4\right)^{1/4} \\
    &\leq (c_3(r,s_1)c_3(r,s_2))^{1/4} \sqrt{ \mathbb{E}\|S_1BS_2^{\top}\|_F^2}.
\end{align*}
Since $S_1$ and $S_2$ are independent, $\mathbb{E}\|S_1B S_2^\top\|_F^2 = \norm{B}_F^2$ by direct calculation, and the result follows.
\end{proof}

\begin{lemma} \label{lem:nuclearLSS} Let $U \in \R^{n\times r}$ be orthonormal and $S \in \R^{t\times n}$ a LSS matrix of $U$ with $t\geq C_0 r\log(2r)$ for a sufficiently large absolute constant $C_0$. Set $q = \lfloor t/2\rfloor,  q' = \lceil t/2\rceil = t-q$, and 
\begin{equation*}
    c_4(r,t):= \frac{\sqrt{r}}{t} +\frac{t-1}{t} (c_3(r,q)c_3(r,q'))^{1/4} \leq \frac{\sqrt{r}}{t}+\sqrt{c_3(r,\lfloor t/2\rfloor)},
\end{equation*} with $c_3$ as in \Cref{concentration_bound}. Then for any symmetric $B \in \R^{n\times n}$ and $\Delta > c_4(r,t)\sqrt{r}$,
\begin{equation*}
    \mathbb{P}(\norm{U^\top S^\top SBS^\top SU}_* < \Delta \norm{B}_*) >1-\frac{c_4(r,t)\sqrt{r}}{\Delta}.
\end{equation*}
\end{lemma}
\begin{proof}
    Write $S^\top S = \frac{1}{t}\sum\limits_{m = 1}^t Y_m$ with $Y_m = e_{i(m)}e_{i(m)}^\top /p_{i(m)}$ where $i(m)\in \{1,...,n\}$ refers to the index chosen on the $m$th draw. Now split $U^\top S^\top SBS^\top SU = \Theta_{diag}+ \Theta_{off}$ into the $m = m'$ and $m \neq m'$ contributions, $\Theta_{diag} = \frac{1}{t^2}\sum\limits_{m} U^\top Y_m B Y_m U$ and $\Theta_{off} = \frac{1}{t^2}\sum\limits_{m \neq m'} U^\top Y_m B Y_{m'}U$.

    For the diagonal part, conditioning on $i(m) =i$ gives $U^\top Y_mBY_{m}U = \frac{B_{ii}}{p_i^2} U_{i,*}^\top U_{i,*}$, a rank-$1$ matrix of nuclear norm $\norm{U_{i,*}}_2^2 = rp_i$; hence
    \begin{equation*}
        \E\norm{\Theta_{diag}}_* \leq \frac{1}{t} \sum_i p_i \frac{B_{ii}}{p_i^2} (rp_i) = \frac{r}{t} \sum_i |B_{ii}| \leq \frac{r}{t}\norm{B}_*.
    \end{equation*}

    For the off-diagonal part, draw $\Omega \subseteq \{1,...,t\}$ uniformly among the subsets of size $q$, independent of $i(1),...,i(t)$, and set $\Omega' = \{1,...,t\} \setminus\Omega$. For every ordered pair $m\neq m'$, one has $\mathbb{P}(m \in \Omega, m' \in \Omega') = \frac{q}{t}\frac{q'}{t-1}$, so taking expectation over $\Omega$ recovers the off-diagonal sum exactly,
    \begin{equation*}
         \E_\Omega[U^\top S_\Omega^\top S_\Omega B S_{\Omega'}^\top S_{\Omega'}U] = \sum\limits_{m\neq m'}\frac{q q'}{t(t-1)}U^\top \frac{1}{q} Y_m B \frac{1}{q'} Y_{m'}U = \frac{t}{t-1}\Theta_{off}, 
    \end{equation*} where $S_\Omega, S_{\Omega'}$ are leverage score sampling matrices with $q,q'$ rows satisfying $S_{\Omega}^\top S_{\Omega} = \frac{1}{q}\sum\limits_{m\in \Omega}Y_m$ and $S_{\Omega'}^\top S_{\Omega'} = \frac{1}{q'}\sum\limits_{m'\in \Omega'}Y_{m'}$, respectively. Conditional on $\Omega$, they are independent, since $\{i(m): m\in \Omega\}$ and $\{i(m): m \in \Omega'\}$ are disjoint subfamilies of independent draws. Using Jensen's inequality, followed by a norm conversion and \Cref{diffterms} with sizes $q,q'$, gives
    \begin{align*}
        \E\norm{\Theta_{off}}_* &\leq \frac{t-1}{t} \E\E_\Omega \norm{U^\top S_\Omega^\top S_\Omega B S_{\Omega'}^\top S_{\Omega'}U}_* \\
        &\leq \frac{t-1}{t}\sqrt{r} \E_\Omega \E \norm{U^\top S_\Omega^\top S_\Omega B S_{\Omega'}^\top S_{\Omega'}U}_F \\
        &\leq \frac{t-1}{t}(c_3(r,q)c_3(r,q'))^{1/4}\norm{B}_F.
    \end{align*} Adding the diagonal and off-diagonal parts, and using $\norm{B}_F \leq \norm{B}_*$, the fact that $c_3$ is nonincreasing, and $q\leq q'$ gives
    \begin{equation*}
        \E\norm{U^\top S^\top SBS^\top SU}_* \leq \left(\frac{r}{t} + \frac{(t-1)\sqrt{r}}{t} (c_3(r,q)c_3(r,q'))^{1/4}\right)\norm{B}_*.
    \end{equation*} Finally, taking Markov's inequality gives the result.
\end{proof}

\begin{theorem}\label{main_leverage}
    Let $A \in \R^{n\times n}$ be a real symmetric matrix, $C \in \R^{n\times r}$ a column subset with $n\geq r$ and thin SVD $C = U\Sigma V^\top$, and $S \in \R^{t\times n}$ a leverage score sampling matrix of $U$ with $t = \bigO(r\log r)$ that is an $\epsilon$-subspace embedding with probability at least $1-\delta$. Define
    \begin{equation*}
        M^*:=\underset{M\in \mathbb{R}^{r\times r}}\argmin\| A-CMC^{\top}\|_F, \qquad \widehat{M}:=\underset{M\in \mathbb{R}^{r\times r}}\argmin\| S(A-CMC^{\top})S^{\top}\|_F.
    \end{equation*}
    Then for all $\Delta > c_4(r,t) \sqrt{r}$,
    \begin{equation*}
        \norm{A - C\widehat{M}C^\top}_* \leq (1+c_5(r,t,\epsilon)\sqrt{r})\norm{A-CM^*C^\top}_*,
    \end{equation*} with probability at least $1-(c_4(r,t)\sqrt{r}/\Delta +\delta)$, where $c_4(r,t)$ is as in \Cref{lem:nuclearLSS}, and $c_5(r,t,\epsilon) = \Delta/((1-\epsilon)^2\sqrt{r})$, exactly as $c_2$ in \Cref{main_gaussian}.
\end{theorem}

\begin{proof}
    The proof is identical to that of \Cref{main_gaussian} with \Cref{lemma32} replaced by \Cref{lem:nuclearLSS} and $c_1$ replaced by $c_4$.
\end{proof}

\section{Numerical Experiments} \label{sec:experiments} In this section, we test \Cref{alg:symmetric_cur} against existing methods on synthetic and real-data examples. We assume the initial column index set $I$ of size $r$ is given, and we compare the following five methods:
\begin{itemize}
    \item \emph{Plain}: $A(\,:\,,I)A(I,I)^\dagger A(\,:\,,I)^\top$,
    \item \emph{Oversample \& Trunc.}: $A(\,:\,,I)A(I,I)_{\lfloor 0.8r\rfloor}^\dagger A(\,:\,,I)^\top$, where the subscript denotes truncation to rank $\lfloor 0.8r\rfloor$ before pseudoinversion\footnote{This is an adaptation, rather than a reproduction, of the indefinite Nystr\"om method proposed in \cite{Nakatsukasa2023}, which oversamples the column set beyond the target rank and then truncates the core matrix back to the target rank. In our setting, the column budget is fixed at exactly $r$, leaving no room for column oversampling. To emulate the stabilizing effect of truncation, we instead truncate the $r\times r$ core matrix to rank $\lfloor 0.8r\rfloor$. The factor $0.8$ is heuristic. The appropriate truncation level depends on the spectrum of $A(I,I)$, and in particular on how many of its eigenvalues can be considered reliable. Consequently, no single fixed fraction is optimal across all matrices. We do not investigate the choice of truncation level further here, since all methods are constrained to the same budget of $r$ columns.},
    \item \emph{Two-sided Proj.}: $A(\,:\,,I)A(\,:\,,I)^\dagger A \, A(\,:\,,I)A(\,:\,,I)^\dagger$,
    \item \emph{Indefinite Nystr\"om} (Gaussian): \Cref{alg:symmetric_cur} with a Gaussian sketch $G$,
    \item \emph{Indefinite Nystr\"om} (LSS): \Cref{alg:symmetric_cur} with a LSS sketch $S$.
\end{itemize}
The initial index set $I$ is obtained via Osinsky's column selection method \cite{osinsky2023closelaa} unless stated otherwise. In all experiments, we use the truncated SVD (TSVD) at rank $r$ as a lower-bound reference and report the relative nuclear-norm error $\frac{\norm{A - \widehat{A}}_*}{\norm{A}_*}$. All experiments were conducted in MATLAB R2024b using double-precision arithmetic.

In \Cref{subsec:effectOS}, we study the effect of oversampling on \Cref{alg:symmetric_cur}. \Cref{subsec:effectcol} compares all five methods on synthetic indefinite matrices with high-quality (Osinsky) and low-quality (uniform) column indices. \Cref{subsec:dataset} tests the methods on real datasets with two indefinite kernels.

\subsection{Effect of Oversampling} \label{subsec:effectOS}
We study the effect of the oversampling ratio $\lfloor t/r \rfloor$ on the accuracy of \Cref{alg:symmetric_cur} using two test matrices:
\begin{enumerate}
    \item $A_1 = V_1\Lambda_1V_1^\top \in \R^{n\times n}$, where $V_1 = \texttt{orth(randn(n))}$ is an incoherent eigenvector matrix and $\Lambda_1$ has diagonal entries decaying geometrically from $1$ to $10^{-14}$ for the first $200$ eigenvalues and equal to $10^{-14}$ for the remainder, with uniformly random signs.
    \item $A_2 = V_2\Lambda_2V_2^\top \in \R^{n\times n}$, where 
    \begin{equation*}
        V_2 = \texttt{orth(diag(randn(n,1))+1e-7*randn(n))}
    \end{equation*} is a coherent eigenvector matrix and $\Lambda_2$ has $50$ eigenvalues equal to $1$, the remainder equal to $10^{-14}$, with uniformly random signs.
\end{enumerate}
We take $n = 1000$, fix $r = 50$, and vary the oversampling ratio $\lfloor t/r \rfloor$ from $1$ to $10$. Matrix $A_2$ is the harder case, as coherent eigenvectors concentrate leverage scores, and the flat eigenvalue profile $(\lambda_1 = \cdots = \lambda_{50} = \pm 1)$ makes cancellations in $A(I,I)$ more likely.

The results are shown in \Cref{fig:oversampling}. For $A_1$ (\Cref{subfig:os1}), both sketches stabilize by oversampling ratio $2$. On the other hand, for $A_2$ (\Cref{subfig:os2}), the coherent structure makes the problem much more sensitive. LSS is highly unstable at low oversampling ratios and requires a ratio of at least $4.5$ to stabilize, which is close to $\log r = \log 50 \approx 4$ and is consistent with the $t = \bigO(r\log r)$ sample size established in \Cref{lem:nuclearLSS}, while Gaussian stabilizes much earlier at around $1.5$-$2$. This is consistent with the theory developed in \Cref{sec:theory}. In both cases, the approximation monotonically improves as the oversampling ratio increases, and both methods ultimately converge to the \emph{Two-sided Proj.} level, which is the best achievable by our algorithm (\Cref{alg:symmetric_cur}) given the column subset $C$.

It is worth emphasizing the appropriate reference point here. The truncated SVD provides a global lower bound, without any restriction to the column set $C$. By contrast, a natural target for any method restricted to the column set $C$ is \emph{Two-sided Proj.}, which computes $CC^\dagger A (C^\top)^\dagger C^\top$ onto $\range(C)$ and is optimal in Frobenius norm among approximations constrained by $C$ on both sides (\Cref{thm:general}). The gap between the other baselines and \emph{Two-sided Proj.} thus reflects a gap in how effectively the available column information is used, and solving the two-sided least-squares problem closes this gap.

\begin{figure}[htp]
    \centering
    \begin{subfigure}[t]{0.465\textwidth}
        \centering
        \includegraphics[width=\linewidth]{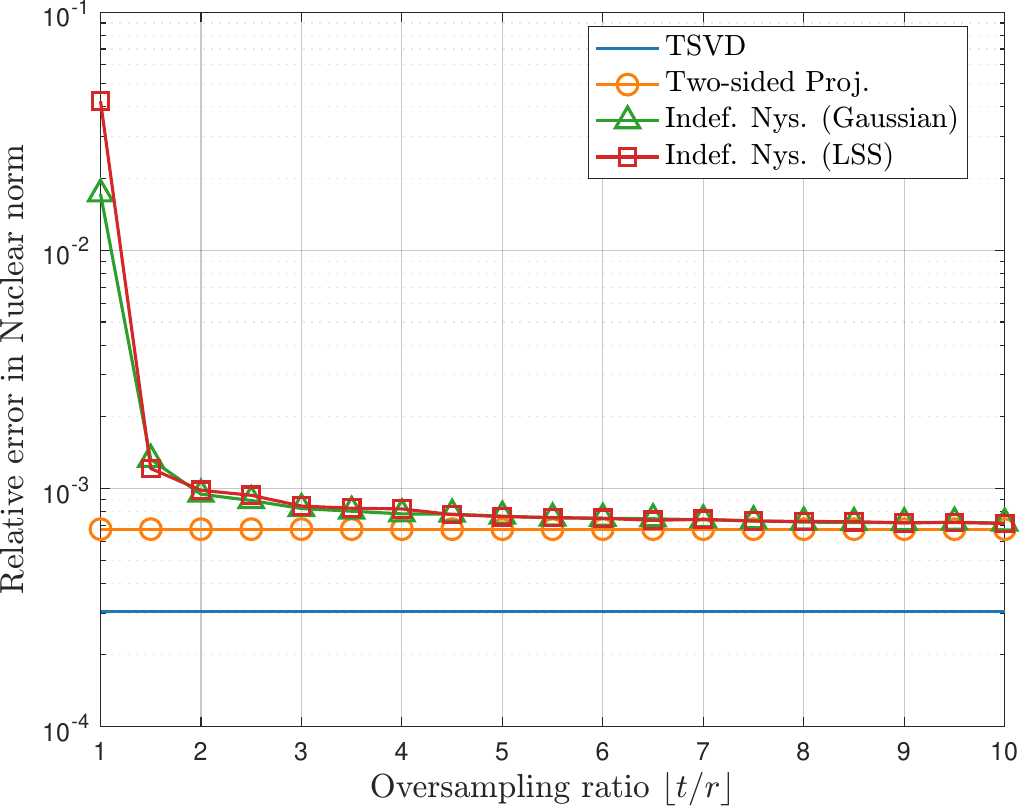}
        \label{subfig:os1}
    \end{subfigure}%
    \begin{subfigure}[t]{0.465\textwidth}
        \centering
        \includegraphics[width=\linewidth]{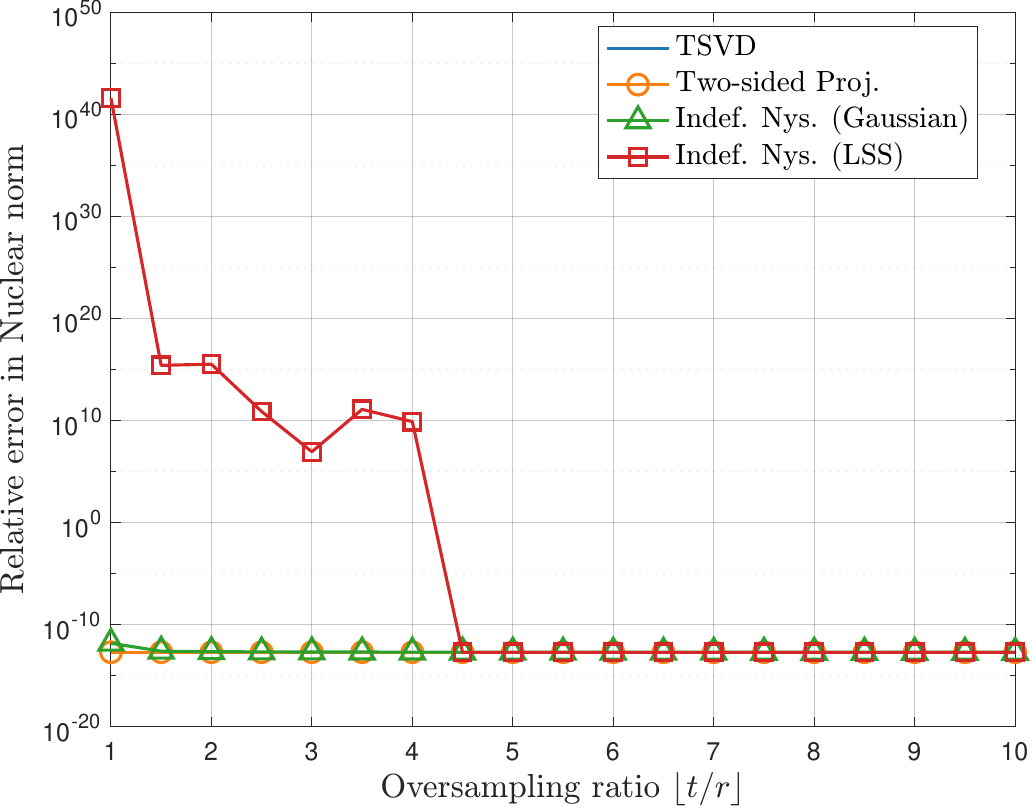}
        \label{subfig:os2}
    \end{subfigure}
    \caption{Effect of oversampling on \Cref{alg:symmetric_cur} ($n = 1000$, $r = 50$). Left: $A_1$, incoherent eigenvectors, geometric eigenvalue decay. Right: $A_2$, coherent eigenvectors, flat spectrum with jump. Gaussian sketch stabilizes at a smaller oversampling ratio than LSS, consistent with \Cref{sec:theory}. The coherent, jump-spectrum case (right) is considerably more challenging: LSS is catastrophically unstable at small oversampling ratios.}
    \label{fig:oversampling}
\end{figure}

\subsection{Effect of initial columns} \label{subsec:effectcol}
We compare all five methods on the matrices $A_1,A_2\in \mathbb{R}^{1000\times 1000}$ defined in \Cref{subsec:effectOS}. We test two column selection strategies: Osinsky's method \cite{osinsky2023closelaa} (high quality) and uniform random sampling (low quality). Here, we take $t = 2r$ for Gaussian and $t = \lfloor 2r\log r\rfloor$ for LSS.

The results are illustrated in \Cref{fig:synthetic}. For high-quality columns (left figures), all methods perform robustly. \emph{Oversample \& Trunc.} has delayed performance due to the core matrix requiring rank truncation proportional to the target rank. For example, in \Cref{subfig:syn21}, it achieves relative error of $\approx 10^{-13}$ only at $r = 70$ as its approximation rank is $70*0.8 = 56$. 

For low-quality columns (right figures), the approximation error is higher for all methods, with \emph{Plain} and \emph{Oversample \& Trunc.} becoming unstable and performing unreliably. Note that \emph{Indef. Nys.} (both Gaussian and LSS) follow \emph{Two-sided Proj.} closely throughout. This shows the robustness of \Cref{alg:symmetric_cur}, especially with LSS, which can be constructed without evaluating every entry of $A$.

\begin{figure}[htp]
    \centering
    \begin{subfigure}[t]{0.465\textwidth}
        \centering
        \includegraphics[width=\linewidth]{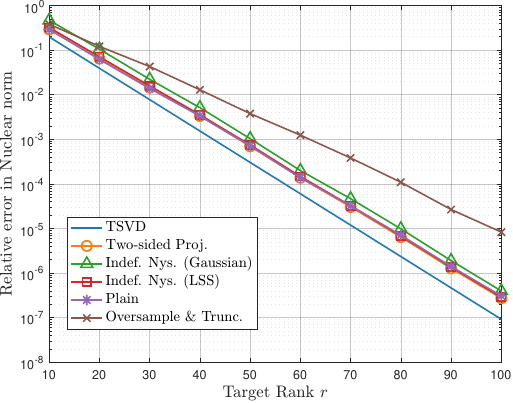}
        \label{subfig:syn11}
        \caption{$A_1$ (incoherent, decay), Osinsky}
    \end{subfigure}%
    \begin{subfigure}[t]{0.465\textwidth}
        \centering
        \includegraphics[width=\linewidth]{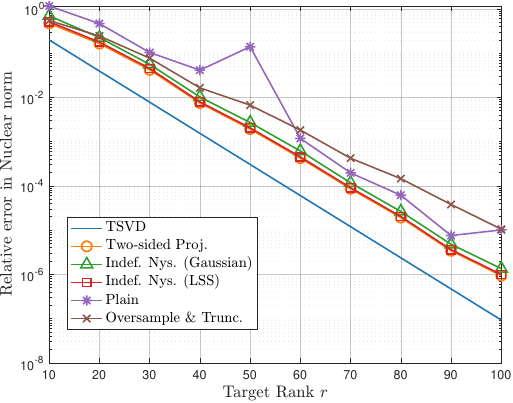}
        \label{subfig:syn12}
        \caption{$A_1$ (incoherent, decay), Uniform}
    \end{subfigure}
    \begin{subfigure}[t]{0.465\textwidth}
        \centering
        \includegraphics[width=\linewidth]{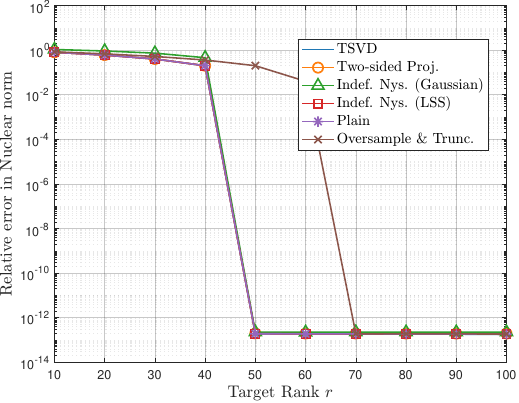}
        \label{subfig:syn21}
        \caption{$A_2$ (coherent, flat + jump), Osinsky}
    \end{subfigure}%
    \begin{subfigure}[t]{0.465\textwidth}
        \centering
        \includegraphics[width=\linewidth]{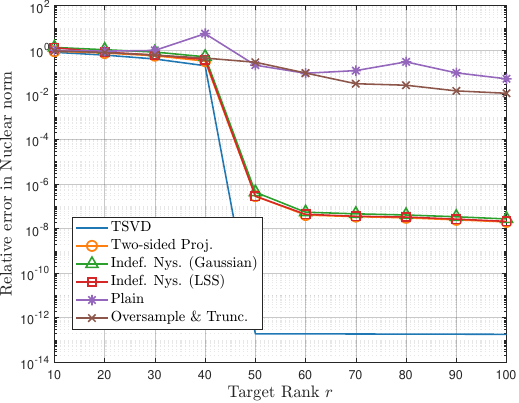}
        \label{subfig:syn22}
        \caption{$A_2$ (coherent, flat + jump), Uniform}
    \end{subfigure}
    \caption{Relative nuclear-norm error $\|A-\widehat{A}\|_*/\|A\|_*$ versus target rank $r$ for synthetic indefinite matrices ($n=1000$). Top row: $A_1$, incoherent eigenvectors with geometric eigenvalue decay. Bottom row: $A_2$, coherent eigenvectors with flat spectrum with jump. Left column: Osinsky's column selection (good quality). Right column: uniform column selection (poor quality). The proposed method is robust throughout, while other strategies may become unstable and behave unpredictably.}
    \label{fig:synthetic}
\end{figure}

\subsection{Dataset with indefinite kernels} \label{subsec:dataset}
The Nystr\"om method is widely used in kernel methods to improve scalability and efficiency. We evaluate the methods on four datasets from LIBSVM \cite{LIBSVM} and the UCI Machine Learning Repository \cite{UCIML}: \texttt{a9a} $(d = 123)$, \texttt{covertype} $(d = 54)$, \texttt{phishing} $(d = 68)$, and \texttt{skin\_nonskin} $(d = 3)$, where $d$ is the feature dimension. For each dataset, we sample $n = 1000$ data points uniformly at random and normalize each feature to have mean zero and variance one. We use two indefinite kernel functions:
\begin{enumerate}
    \item Thin plate spline: $k(x,y) = \norm{x - y}^2\log\norm{x-y}$;
    \item Sigmoid: $k(x,y) = \tanh{\left(\frac{x^\top y}{d}\right)}$.
\end{enumerate} 
The results for the sigmoid and thin plate spline kernels are shown in \Cref{fig:dataset_sigmoid,fig:dataset_thinplate}, respectively. In both cases, the column index set is obtained via Osinsky's method. We omit \emph{Indef. Nys.} (Gaussian) from these figures because forming $GAG^\top$ requires access to all $n^2$ entries of $A$.

\paragraph{Sigmoid kernel (\Cref{fig:dataset_sigmoid})} Across all four datasets, \emph{Indef. Nys.} (LSS) and \emph{Two-sided Proj.} track the TSVD lower bound closely. \emph{Plain} exhibits erratic, non-monotone errors on \texttt{a9a} and \texttt{phishing}, where $A(I,I)$ underestimates the eigenvalues of $A$ at certain target ranks. \emph{Oversample \& Trunc.} mitigates the instability but remains worse than the proposed method, as it truncates the core matrix to a smaller rank. The \texttt{skin\_nonskin} dataset achieves the fastest spectral decay and all methods follow the spectral decay quite robustly.

\paragraph{Thin plate spline kernel (\Cref{fig:dataset_thinplate})}
The thin plate spline kernel is less low-rank and approximately half the eigenvalues are negative, so errors are higher overall. Despite this, the qualitative picture mirrors the sigmoid case: \emph{Indef. Nys.} (LSS) and \emph{Two-sided Proj.} consistently perform robustly, while \emph{Plain} can be unstable and \emph{Oversample \& Trunc.} exhibits higher errors. Across both kernels and all datasets, \emph{Indef. Nys.} (LSS) performs at least as well as \emph{Two-sided Proj.} while requiring only $t^2$ matrix-entry evaluations after the initial column set is identified, in contrast to the full $n^2$ entry evaluations required by \emph{Two-sided Proj.} and Gaussian sketching.

\begin{figure}[htp]
    \centering
    \begin{subfigure}[t]{0.465\textwidth}
        \centering
        \includegraphics[width=\linewidth]{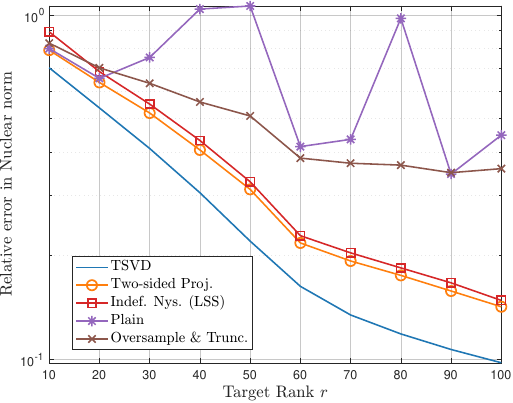}
        \caption{\texttt{a9a}}
        \label{subfig:a9a_sigmoid}
    \end{subfigure}%
    \begin{subfigure}[t]{0.465\textwidth}
        \centering
        \includegraphics[width=\linewidth]{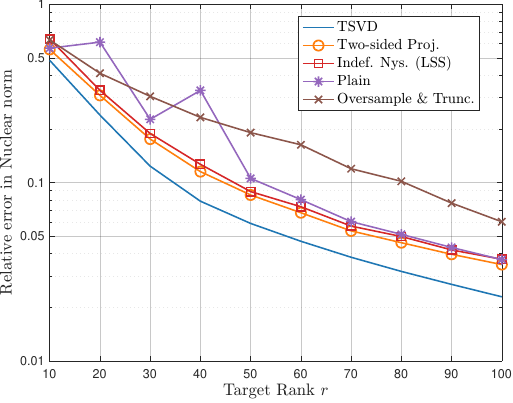}
        \caption{\texttt{covertype}}
        \label{subfig:covertype_sigmoid}
    \end{subfigure}
    \begin{subfigure}[t]{0.465\textwidth}
        \centering
        \includegraphics[width=\linewidth]{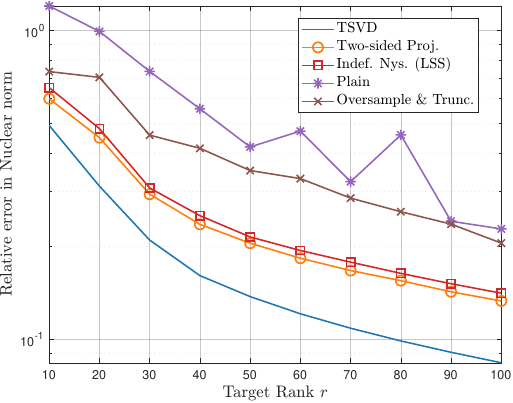}
        \caption{\texttt{phishing}}
        \label{subfig:phishing_sigmoid}
    \end{subfigure}%
    \begin{subfigure}[t]{0.465\textwidth}
        \centering
        \includegraphics[width=\linewidth]{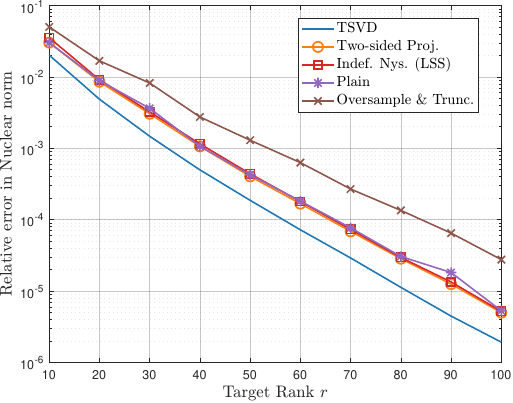}
        \caption{\texttt{skin\_nonskin}}
        \label{subfig:skinnonskin_sigmoid}
    \end{subfigure}
    \caption{Relative nuclear-norm error $\|A-\widehat{A}\|_*/\|A\|_*$ versus target rank
      $r$ for the sigmoid kernel on four datasets
      ($n=1000$, Osinsky column selection). The indefinite Nystr\"om method (LSS) and Two-sided
      Proj.\ track the TSVD lower bound throughout. The plain Nystr\"om method is unstable on
      \texttt{a9a}, \texttt{covertype}, and \texttt{phishing}.}
    \label{fig:dataset_sigmoid}
\end{figure}

\begin{figure}[htp]
    \centering
    \begin{subfigure}[t]{0.465\textwidth}
        \centering
        \includegraphics[width=\linewidth]{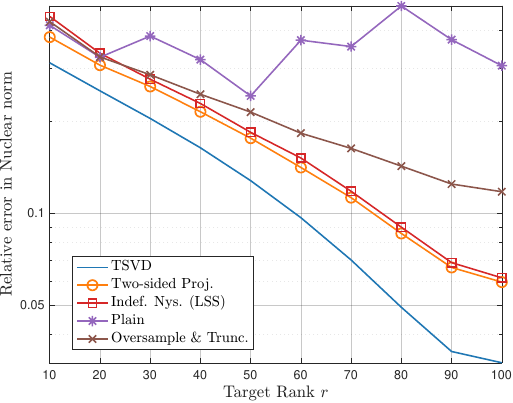}
        \caption{\texttt{a9a}}
        \label{subfig:a9a_thinplate}
    \end{subfigure}%
    \begin{subfigure}[t]{0.465\textwidth}
        \centering
        \includegraphics[width=\linewidth]{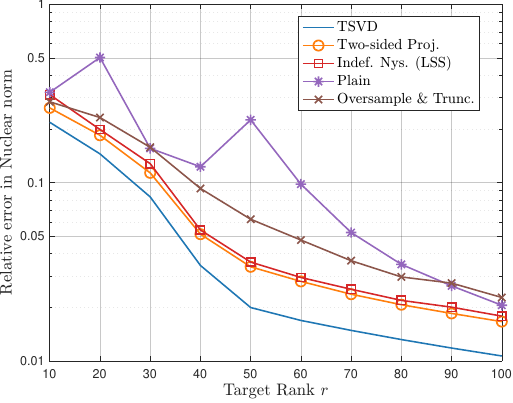}        
        \caption{\texttt{covertype}}
\label{subfig:covertype_thinplate}                
    \end{subfigure}
    \begin{subfigure}[t]{0.465\textwidth}
        \centering
        \includegraphics[width=\linewidth]{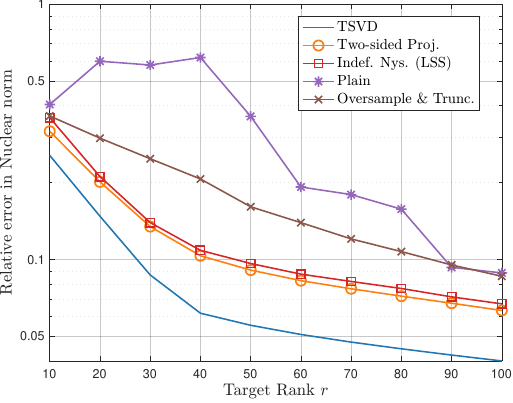}
        \caption{\texttt{phishing}}
        \label{subfig:phishing_thinplate}
    \end{subfigure}%
    \begin{subfigure}[t]{0.465\textwidth}
        \centering
        \includegraphics[width=\linewidth]{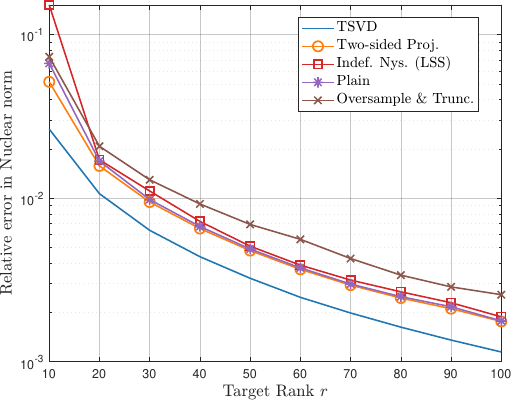}
        \caption{\texttt{skin\_nonskin}}
        \label{subfig:skinnonskin_thinplate}
    \end{subfigure}
    \caption{Same experiment as \Cref{fig:dataset_sigmoid} but with the thin plate
      spline kernel. The kernel is less low-rank
      than the sigmoid, so errors are higher overall. The indefinite
      Nystr\"om method (LSS) consistently matches Two-sided Proj. and avoids the
      instability of the plain Nystr\"om method on \texttt{a9a}, \texttt{covertype}, and \texttt{phishing}.}
    \label{fig:dataset_thinplate}
\end{figure}

\section{The use of AI assistance}
Parts of this project were carried out with assistance from ChatGPT-5.6-Sol. The contribution was to the leverage score analysis of \Cref{sec:lssampling}. Our first proof of the nuclear-norm bound partitioned the $t$ samples into $\lceil\sqrt r\,\rceil$ blocks and bounded the sampled quantity block by block, treating the samples within a block as mutually dependent; this led to a loose analysis and required a sketch size of $t=\mathcal{O}(r^{3/2}\log r)$. Prompted to look for a sharper argument, ChatGPT proposed a decoupling approach, which isolated the $t$ diagonal sample pairs and decoupled all $t(t-1)$ off-diagonal pairs at once through a random balanced partition. This removes the factor of $\sqrt r$ and reduces the requirement to $t=\mathcal{O}(r\log r)$, matching the typical theory for LSS in the literature \cite{mahoneya2009CUR}; it is the argument now given in \Cref{lem:nuclearLSS}. ChatGPT also helped to polish the exposition.

The authors verified every definition, lemma, and proof independently, and confirmed the cited sources; responsibility for any errors, and for the manuscript as a whole, rests entirely with the authors.

\bibliographystyle{siamplain}
\bibliography{references}

\end{document}